\documentclass[11pt]{amsart}

\usepackage{amssymb,amsmath,amscd,amsthm,enumerate,verbatim}
\usepackage[utf8]{inputenc}
\usepackage[mathscr]{euscript}
\usepackage{bbm}
\usepackage{graphicx}
\usepackage{color}

\date{}

\mathcode`\@="8000                      
{\catcode`\@=\active \gdef@{\mkern1mu}} 

\newcommand{\B}{{\mathbb B}}
\newcommand{\Hi}{{\mathcal H}}
\newcommand{\Z}{{\mathbbm Z}}
\newcommand{\R}{{\mathbbm R}}
\newcommand{\C}{{\mathbbm C}}
\newcommand{\CC}{{\mathcal C}}

\newcommand{\I}{{\mathbbm I}}

\newcommand{\Int}{{\mathrm{int}}}

\newcommand{\e}{{\varepsilon}}

\newtheorem{theorem}{Theorem}[section]
\newtheorem{lemma}[theorem]{Lemma}
\newtheorem{prop}[theorem]{Proposition}
\newtheorem{coro}[theorem]{Corollary}

\theoremstyle{definition}
\newtheorem{remark}[theorem]{Remark}
\theoremstyle{definition}

\theoremstyle{definition}

\theoremstyle{definition}

\theoremstyle{definition}

\theoremstyle{definition}

\allowdisplaybreaks

\numberwithin{equation}{section}

\newcommand{\tr}{\mathrm{tr} }

\newcommand{\p}{{\mathfrak p}}

\newcommand{\fr}{{\mathfrak{r}}}

\newcommand{\vect}{{\mathrm{vec}}}

\newcommand{\set}[1]{\left\{#1\right\}}
\newcommand{\eqdef}{\overset{\mathrm{def}}=}

\begin{document}

\author[M.\ Embree]{Mark Embree}

\address{Department of Mathematics, Virginia Tech, 225 Stanger Street, Blacksburg, VA~24061, USA}

\email{embree@vt.edu}

\thanks{M.\ E.\ was supported in part by National Science Foundation grant DGE-1545362.}

\author[J.\ Fillman]{Jake Fillman}

\address{Department of Mathematics, Virginia Tech, 225 Stanger Street, Blacksburg, VA~24061, USA}

\email{fillman@vt.edu}

\thanks{J.\ F.\ was supported in part by an AMS-Simons Travel Grant, 2016--2018}


\title[Spectra of Small Periodic Potentials]{Spectra of Discrete Two-Dimensional Periodic Schr\"odinger Operators with Small Potentials}

\maketitle

\begin{abstract}
We show that the spectrum of a discrete two-dimensional periodic Schr\"odinger operator on a square lattice with a sufficiently small potential is an interval, provided the period is odd in at least one dimension. In general, we show that the spectrum may consist of at most two intervals and that a gap may only open at energy zero. This sharpens several results of Kr\"uger and may be thought of as a discrete version of the Bethe--Sommerfeld conjecture. We also describe an application to the study of two-dimensional almost-periodic operators.
\end{abstract}


\section{Introduction}

Researchers in mathematics and physics have extensively investigated spectral and quantum dynamical characteristics of one-dimensional Hamiltonians of the form
\begin{equation} \label{eq:1dso}
[H\psi]_n
=
\psi_{n-1} + V_n \psi_n + \psi_{n+1},
\quad
n \in \Z, \; \psi \in \ell^2(\Z),
\end{equation}
where $V:\Z \to \R$ is a bounded sequence, known as the \emph{potential}. The most heavily-studied models are those for which $V$ is periodic, almost-periodic, or random. Almost-periodic operators can exhibit wild spectral characteristics, such as Cantor spectrum of zero Lebesgue measure and purely singular continuous spectral type. The literature on such operators is vast; see \cite{DamanikESOSurv,DF,JitoSurv,JMSurv} and references therein. Though such phenomena were once thought to be exotic and rare, Cantor spectrum and purely singular continuous spectral type turn out to be generic in a rather robust sense for many families of one-dimensional operators having the form~\eqref{eq:1dso} \cite{Avi09, AviDam2005DMJ, AviJit2009Annals, DamLen2006DMJ, Simon1995Annals}. 
The more complicated structure of higher-dimensional analogs of \eqref{eq:1dso} makes such models prohibitively difficult to study, even in simple cases. 
With some notable exceptions (see, e.g.\ \cite{Gordon2015,GordNem2016,KarpLee2013,KLSS2016}), spectral properties of aperiodic almost-periodic Schr\"odinger operators in higher dimensions have proved quite difficult to study.

Recently some success has been achieved by studying operators that are separable, in the sense that they can be separated into a sum of two commuting one-dimensional operators; such separable operators are amenable to attack, as their spectra can be expressed as the sum of the spectra of their one-dimensional components, which are well-understood. Even in this situation, one must deal with delicate challenges, such as arithmetic sums of Cantor sets and convolutions of singular measures. Initial insight about these operators and their spectra came from numerical studies, mainly appearing in the physics literature \cite{DEG15, EveLif2006, EveLif2008, EveLifPreprint, IlaLibEveLif2004, Sire1989, SirMos1989, SirMos1990, SirMosSad1989, ThiSch2013}. Rigorous results have been obtained fairly recently in \cite{DGS13,FTY2016}.

The present paper addresses discrete two-dimensional Schr\"odinger operators on a square lattice, defined by
\begin{equation} \label{eq:2dso}
\begin{split}
H = \Delta + V, \quad & [V\psi]_{n,m} = V_{n,m}\psi_{n,m} \\
[\Delta\psi]_{n,m}
=
\psi_{n-1,m} + \psi_{n+1,m} & + \psi_{n,m-1} + \psi_{n,m+1}, \\
n,m \in\Z,
\; \psi & \in \ell^2(\Z^2),
\end{split}
\end{equation}
with $V$ periodic in the sense that there exist $p,q \in \Z_+$ with 
\begin{equation} \label{eq:2dperpot:def}
V_{n+p,m} = V_{n,m+q}
= 
V_{n,m}
\text{ for all }n,m \in \Z.
\end{equation}
When \eqref{eq:2dperpot:def} holds for some $\p = (p,q) \in \Z_+^2$, we say $V$ is $\p$-\emph{periodic}. The study of Schr\"odinger operators on $\Z^d$ (and more generally on $\Z^d$-periodic lattices) is of interest due to applications in chemistry and physics; see the survey \cite{CGPNG} for instance. Many papers and books have been written about operators on graphs; see \cite{BroHae2012, Chung1997, CDGT1988, CDS1995, GKT1993, Post2012} and references therein.

 Our main result shows that the spectra of such objects are quite different from those of operators like \eqref{eq:1dso}. Concretely, we prove that any periodic potential in dimension two that is sufficiently small will produce a spectrum with at most two connected components if $p$ and $q$ are both even, and with one connected component otherwise. This result contrasts strongly with the one-dimensional case, in which a generic $p$-periodic operator has spectrum with $p$ connected components.

\begin{theorem} \label{t:discreteBSconj}
Let $\p = (p,q)$ be given. There exists a constant $C = C_\p > 0$ such that the following statements hold true:
\begin{enumerate}
\item[{\rm(1)}] If $V$ is $\p$-periodic and $\|V\|_\infty \le C$, then $\sigma(H_V)$ has at most two connected components. 

\item[{\rm(2)}] If at least one of $p$ or $q$ is odd, then $\sigma(H_V)$ is a single interval whenever $V$ is $\p$-periodic and $\|V\|_\infty \le C$. 
\end{enumerate}
\end{theorem}

This result can be regarded as a discrete version of the Bethe--Sommerfeld conjecture (in dimension two), which posits that the spectrum of the operator
$-\nabla^2 + V$
acting in $L^2(\R^d)$ ($d \geq 2$) contains a half-line whenever $V$ is periodic in the sense that there exists a rank-$d$ lattice $\Lambda \subset \R^d$ such that
\[
V(x+\gamma) = V(x)
\text{ for all } x \in \R^d, \gamma \in \Lambda.
\]
In particular, the analysis of the discrete operator $H = \Delta + \lambda V$ with $\lambda$ small mirrors that of the high-energy regime of the (unbounded) operator $-\nabla^2 + V$ in $L^2(\R^d)$. The Bethe--Sommerfeld conjecture has inspired intense study, with substantial contributions from many authors, including (but certainly not limited to) \cite{HelMoh98,Karp97,PopSkr81,Skr79,Skr84,Skr85,Vel88}, and culminating in the paper of Parnovskii \cite{Parn2008AHP}. However, our proof techniques here are  a bit different than those used in the continuum setting. In particular, we employ a pair of soft arguments: one to count eigenvalues, and one to prove that the eigenvalue counts forbid small potentials from opening gaps at nonzero energies. 
These soft arguments must be refined on a finite exceptional set using perturbation theory for degenerate eigenvalues, showing that gaps cannot form at such energies.

On the discrete side, Kr\"uger proved part (2) of Theorem~\ref{t:discreteBSconj} under the more restrictive assumption $\mathrm{gcd}(p,q) = 1$ \cite{KrugPreprint}. He also constructed examples with $\p = (2,2)$ for which $\sigma(\Delta+V)$ contains two intervals for arbitrarily small $V$. In fact, with
\[
V^\delta_{n,m}
=
\delta(-1)^{n+m},
\quad
\delta > 0, \; n,m\in\Z,
\]
he shows that $\sigma(\Delta+V^\delta)\cap (-\delta,\delta) = \emptyset$. Thus, our result improves the result of \cite{KrugPreprint} to incorporate the optimal range of validity vis-\`a-vis arithmetic conditions on $\p$. Moreover, our proof is substantially simpler than Kr\"uger's proof of \cite[Theorem~6.1]{KrugPreprint}, as he uses some sophisticated algebraic tools (cf.\ \cite[Section~5]{KrugPreprint}). 
Finally, in the course of the proof, we answer Questions~6.2 and~6.4 in \cite{KrugPreprint}.
Question~6.2 asks for optimal conditions on $p$ and $q$ so that the conclusion of part~2 of Theorem~\ref{t:discreteBSconj} holds; we prove that $\gcd(p,q)$ odd suffices. Question~6.4 asks whether there exists another mechanism by which one may open gaps in higher dimensions at small coupling; our arguments answer this question in the negative.
\medskip

One immediate consequence of Theorem~\ref{t:discreteBSconj} is that it is much more difficult to produce Cantor spectrum in high dimensions. For example, Theorem~\ref{t:discreteBSconj} immediately implies that if a sequence of periodic potentials converges sufficiently rapidly, then the spectrum of the resulting limit-periodic operator can have at most two connected components. Again, this draws a strong contrast with one-dimensional limit-periodic operators, which generically exhibit zero-measure Cantor spectrum \cite{Avi09,DamFilLuk15,FilOng16}.

\begin{coro} \label{coro:lp:application}
Suppose $\p_j$ is a sequence of periods such that $\p_j | \p_{j+1}$ for all $j$ {\rm(}in the sense that each component of $\p_j$ divides the corresponding component of $\p_{j+1}${\rm)}. There exist $\delta_j > 0$ with the following property. If $V_j$ is a $\p_j$-periodic potential with $\|V_j\|_\infty \leq \delta_j$ for all $j$, and
\[
V
=
\sum_{j=1}^\infty V_j,
\]
then $\sigma(\Delta + V)$ consists of at most two intervals. If at least one coordinate of $\p_j$ is odd for every $j$, then $\sigma(\Delta+V)$ is an interval.
\end{coro}

\begin{proof}
This follows from Theorem~\ref{t:discreteBSconj} by repeating the arguments that prove \cite[Theorem~7.1]{KrugPreprint} verbatim.
\end{proof}

In Section~\ref{sec:periodicreview}, we recall some necessary facts about discrete periodic operators, which we then use to prove Theorem~\ref{t:discreteBSconj} in Section~\ref{sec:discreteBSC}.

\section*{Acknowledgements}
The authors thank George Hagedorn for helpful conversations about this work.  J.F.\ is grateful to the Simons Center for Geometry and Physics for their hospitality during the program ``Between Dynamics and Spectral Theory'', during which portions of this work were completed. J.F.\ also thanks Robert Israel for a helpful and insightful answer to a relevant question on Mathematics Stack Exchange.

\section{Discrete Periodic Operators: A Brief Review} \label{sec:periodicreview}
Let us briefly review the relevant spectral characteristics of discrete periodic operators. In our arguments, we will need some particular facts about the discrete one-dimensional Laplacian, so we begin by collecting those.

\subsection{The Discrete Laplacian in Dimension One}

The discrete Laplacian on $\ell^2(\Z)$ is defined by
\[
[\Delta u]_n
=
u_{n-1} + u_{n+1},\quad
n \in \Z, \; u \in \ell^2(\Z).
\]
The analysis that follows comes from viewing $\Delta$ as a periodic Jacobi matrix; for more thorough discussions of periodic Jacobi matrices, see \cite{DF,simszego,teschljacobi}. Given $r \in \Z_+$ ($r\geq 3$) and $\theta \in \R$, denote by $\Delta^r_\theta$ the self-adjoint matrix
\[
\Delta_\theta^r
=
\begin{bmatrix}
0 & 1 &&& e^{-i\theta} \\ 
1 & 0 & 1 && \\
& \ddots & \ddots & \ddots & \\
&& 1 & 0 & 1 \\
e^{i\theta}  &&& 1 & 0
\end{bmatrix}
\in \C^{r\times r}.
\]
For $r = 1,2$, one has to be a little careful, defining
\[
\Delta_\theta^1
=
2\cos\theta,
\quad
\Delta_\theta^2
=
\begin{bmatrix}
0 & 1+e^{-i\theta} \\
1+e^{i\theta} & 0
\end{bmatrix}.
\]

\begin{prop} \label{p:1dlaplacianeigmults}
Let $r \in \Z_+$ be given. Then,
\begin{align*}
\sigma(\Delta_0^r)
& =
\set{2\cos\left(\frac{\pi j}{r}\right) : 0 \leq j \leq r \text{ and } j \text{ is even }} \\[.25em]
\sigma(\Delta_{\pi/2}^r)
& =
\set{2\cos\left(\frac{\pi j}{2r}\right) : 0 \leq j \leq 2 r \mbox{ and $j$ is odd }} \\[.25em]
\sigma(\Delta_\pi^r)
& =
\set{2\cos\left(\frac{\pi j}{r}\right) : 0 \leq j \leq r \text{ and } j \text{ is odd }}.
\end{align*}
For $\Delta_0^r$ and $\Delta_\pi^r$,  the eigenvalues $\pm 2$ are simple; the other eigenvalues all have multiplicity two.
All eigenvalues of $\Delta_{\pi/2}^r$ are simple.
\end{prop}

\begin{proof}
For each $j$ with $0 \leq j \leq r$, define the vectors $\vec v^{@@\pm}(j)$ by
\[
\vec v^{@@\pm}_k(j) = e^{\pm\frac{ijk\pi}{r}},
\quad
1 \le k \le r.
\]
One can readily verify that $\vec v^{@@\pm}(j)$ is an eigenvector of $\Delta_0^r$ for even $j$ and of $\Delta_\pi^r$ for odd $j$, corresponding to the eigenvalue $2\cos(\pi j/r)$. Moreover, for $0 < j < r$, $\vec v^{@+}(j)$ and $\vec v^{@-}(j)$ are linearly independent, which gives the desired statements on multiplicities of the eigenvalues of $\Delta_0^r$ and $\Delta_\pi^r$.

For $\Delta_{\pi/2}^r$ and $1 \le j \le r$, define $\theta_j = (-1)^j\pi(2j-1)/(2r)$, and put 
\[
\vec w_k(j) = e^{-i k \theta_j}
\quad
1 \le k \le r.
\]
One can check that $\vec w(j)$ is an eigenvector of $\Delta_{\pi/2}^r$ corresponding to the eigenvalue $2\cos(\theta_j)$ for each $1 \le j \le r$. These $r$ distinct points are precisely the eigenvalues given in the proposition.
\end{proof}

To handle exceptional energies in arguments that follow, we will use a perturbative analysis that involves the derivatives of the eigenvalues of $\Delta_\theta^r$ with respect to $\theta$.

\begin{lemma} \label{lem:cornerderivatives}
Fix $r \in \Z_+$, and denote the eigenvalues of $\Delta_\theta^r$ by
\[
\lambda_1(\theta)
\leq
\cdots
\leq 
\lambda_r(\theta).
\]
For every $1 \le j \le r$:
\begin{enumerate}
\item[{\rm (a)}] $\lambda_j$ is right-differentiable at $0$ and left-differentiable at~$\pi$;
\item[{\rm (b)}] for all $\theta \in (0,\pi)$, $\lambda_j$ is differentiable and 
$(-1)^{r-j}\lambda_j'(\theta) < 0$;
\item[{\rm (c)}] for $\phi \in \{0,\pi\}$, 
\begin{equation} \label{eq:eigderivs}
|\lambda_j'(\phi)|
=
\frac{1}{r}\sqrt{4-\lambda_j(\phi)^2}.
\end{equation}
\end{enumerate}
\end{lemma}

\begin{proof}
That $\lambda_j$ is differentiable (even real-analytic) on $(0,\pi)$ with $(-1)^{r-j} \lambda_j' < 0$ thereupon is well-known \cite[Theorem~5.3.4]{simszego}. Moreover, by general eigenvalue perturbation theory, it is known that $\lambda_j$ enjoys a continuously differentiable extension through the points $0$ and $\pi$; see, e.g.\ \cite[Theorem~II.6.8]{Kato}. Thus, we need only concentrate on proving \eqref{eq:eigderivs}. We will prove this in the case when $r$ is even. The proof for odd $r$ is identical, except $-2 = \lambda_1(\pi)$ instead of $\lambda_1(0)$.
\medskip

Let $D$ denote the associated discriminant, defined by
\[
D(z)
=
\tr(T_z^{@r}),
\quad
T_z = \begin{bmatrix}
z & - 1 \\ 1 & 0
\end{bmatrix},
\quad z \in \C.
\]
($T_z^{@r}$ denotes the $r$th power of the matrix $T_z$.)
Given a normalized eigenvector $\vec w$ of $\Delta_\theta^r$ corresponding to the eigenvalue $\lambda_j(\theta)$ of $\Delta_\theta^r$, it is straightforward to verify that
\[
T_{\lambda_j(\theta)}^{@r} \begin{bmatrix} \vec w_2 \\ \vec w_1 \end{bmatrix}
=
e^{i\theta}\begin{bmatrix} \vec w_2 \\ \vec w_1 \end{bmatrix},
\]
and hence, since $\det(T_z)=1$,
\begin{equation} \label{eq:Dlambdaj}
D(\lambda_j(\theta)) = 2\cos\theta
\text{ for all }
\theta \in [0,\pi].
\end{equation}
By a straightforward induction, one can check that
\begin{equation}\label{eq:Dchebyshev}
D(2\cos\eta) = 2\cos(r\eta)
\text{ for every } \eta \in [0,\pi].
\end{equation}
Concretely, it is easy to verify that \eqref{eq:Dchebyshev} holds when $r = 1,2$. Inductively, if \eqref{eq:Dchebyshev} holds for $r$ and $r-1$, then, by the Cayley--Hamilton theorem, 
\begin{align*}
\tr(T_{2\cos(\eta)}^{r+1})
&=
2\cos(\eta) \tr(T_{2\cos(\eta)}^r) - \tr(T_{2\cos(\eta)}^{r-1}) \\
& =
4\cos(\eta) \cos(r\eta) - 2 \cos ((r-1)\eta) \\
& =
2\cos((r+1)\eta).
\end{align*}
In view of \eqref{eq:Dchebyshev}, every point of the form $2\cos(\pi m/r)$ with $0 < m < r$ an integer is a critical point of $D$. Hence, every eigenvalue of $\Delta_0$ or $\Delta_\pi$ except $\pm 2$ is a critical point of $D$. Differentiate both sides of \eqref{eq:Dlambdaj} twice (with respect to $\theta$) to obtain
\[
D''(\lambda_j(\theta))\lambda_j'(\theta)^2 + D'(\lambda_j(\theta)) \lambda_j''(\theta)
=
-2\cos\theta.
\]
Since, when $j\ne 1, r$,  $\lambda_j(0)$ is a critical point of $D$, we deduce
\[
D''(\lambda_j(0)) \lambda_j'(0)^2 = -2.
\]
Consequently,
\[
|\lambda_j'(0)|
=
\sqrt{-\frac{2}{D''(\lambda_j(0))}},
\]
for $1<j<r$. Similarly,
\[
|\lambda_j'(\pi)|
=
\sqrt{\frac{2}{D''(\lambda_j(\pi))}}
\]
for all $1 \leq j \leq r$. Thus, we need to compute $D''$ at the critical points of $D$. Differentiate \eqref{eq:Dchebyshev} twice with respect to $\eta$ and plug in $\eta = \pi m/r$ with $1 \leq m \leq r-1$ an integer to get
\[
D''\left(2\cos\left(\frac{\pi m}{r}\right)\right)
=
(-1)^{m+1}\frac{r^2}{2\sin^2(\pi m/r)}.
\]
Thus, we obtain \eqref{eq:eigderivs} for $\phi=0$ and $1 < j < r$, as well as for $\phi = \pi$ and $1\le j\le r$.

 It remains to check the derivative at the eigenvalues $\pm 2$: for even $r$, this amounts to showing that $\lambda_1'(0) = \lambda_r'(0)=0$.  We can explicitly compute the derivative at those points using first-order perturbation theory for simple eigenvalues. Concretely, $\vec w_j \equiv 1$ supplies an eigenvector of $\Delta_0^r$ corresponding to the eigenvalue~$2$. An explicit calculation gives

\[
\frac{\partial}{\partial \theta} \Delta_\theta^r
=
i \left( e^{i\theta}\vec e_r^{} \vec e_1^\top - e^{-i\theta}\vec e_1^{} \vec e_r^\top \right),
\]
so, by the Feynman--Hellmann theorem (see \cite[Theorem~1.4.7]{SimCCA4} or \cite[Chapter~II]{Kato}), we get
\[
\lambda_r'(0) 
= 
i\vec w^\top 
\left( \vec e_r^{} \vec e_1^\top - \vec e_1^{} \vec e_r^\top \right)
\vec w
= 
0.
\]
Similar considerations work for the eigenvalue $-2$ using $\vec u_j = (-1)^j$, which is an eigenvector of $\Delta_0^r$ since $r$ is even. When $r$ is odd, the proof is identical except that $\vec u$ is an eigenvector of $\Delta_\pi^r$ instead of $\Delta_0^r$.
\end{proof}

\begin{remark}
The identity \eqref{eq:Dchebyshev} shows that $D$ is a (rescaled) Chebyshev polynomial. This is a special case of a more general fact for periodic Jacobi matrices; see, e.g.\ \cite[Example~5.7.3]{simszego}.
\end{remark}

\subsection{Periodic Operators in Dimension Two} We briefly recall the main tools that we will need; for a more complete review and  enjoyable reading, see \cite{KrugPreprint,Kuchment2016BAMS}. We consider operators $H$ of the form \eqref{eq:2dso} where the potential is $\p =(p,q)$-periodic in the sense of \eqref{eq:2dperpot:def}. In this situation, we can compute the spectrum of $H$ using a direct integral decomposition as follows. Let $\Gamma$ denote the fundamental domain
\[
\Gamma
=
\Gamma_\p
\eqdef 
\big([0,p)\times [0,q) \big) \cap \Z^2.
\]
The fibers of the direct integral are given by
\[
\Hi_\Gamma 
= 
\C^\Gamma
\eqdef
\set{ (\psi_{n,m}) : \psi_{n,m} \in \C \text{ for each } (n,m) \in \Gamma}.
\] 
For each $\theta, \varphi \in [0,\pi]$, let $H_{\theta,\varphi}^\Gamma$ be the operator given by restricting $H$ to $\Hi_\Gamma$ with boundary conditions of phase $\theta$ at the vertical boundaries and phase $\varphi$ on the horizontal boundaries. Concretely,
\[
\left[H_{\theta,\varphi}^\Gamma \psi \right]_{n,m}
=
\left[ H \psi^{\theta,\varphi} \right]_{n,m},
\quad
(n,m) \in \Gamma,
\]
where $\psi^{\theta,\varphi}:\Z^2 \to \C$ is defined by the conditions $\psi^{\theta,\varphi}_{n,m} = \psi_{n,m}$ if $(n,m) \in \Gamma$, and 
\begin{equation} \label{eq:2dBCs}
\psi^{\theta,\varphi}_{n+p,m}=
e^{i\theta} \psi^{\theta,\varphi}_{n,m},
\qquad
\psi^{\theta,\varphi}_{n,m+q}=
e^{i\varphi} \psi^{\theta,\varphi}_{n,m}
\quad
\text{for all } n,m \in \Z.
\end{equation}
Equivalently, one can define the space
\[
\ell^\infty_{\Gamma,\theta,\varphi}(\Z^2)
=
\set{\psi \in \ell^\infty(\Z^2) : \text{ \eqref{eq:2dBCs} holds}},
\]
which is isomorphic to $\mathcal H_\Gamma$ in a canonical fashion. Under this isomorphism, $H_{\theta,\varphi}^\Gamma$ coincides with the restriction of $H$ to $\ell^\infty_{\Gamma,\theta,\varphi}(\Z^2)$.

We denote the eigenvalues of $H_{\theta,\varphi}^\Gamma$ (counted with multiplicity) by
\[
\lambda_1(\theta,\varphi)
\leq 
\lambda_2(\theta,\varphi)
\leq
\cdots
\leq
\lambda_{pq}(\theta,\varphi).
\]
Together, the spectra of the $H_{\theta,\varphi}^\Gamma$ operators give a nice characterization of the spectrum of $H$.

\begin{theorem} \label{t:2dperspec}
If $V$ is $\p$-periodic, then
\[
\sigma(H)
=
\bigcup_{\theta,\varphi \in [0,\pi]} \sigma\!\left( H^{\Gamma}_{\theta,\varphi} \right)
=
\bigcup_{j=1}^{pq} B_j,
\]
where $\Gamma = \Gamma_\p$ and
\[
B_j
=
\set{\lambda_j(\theta,\varphi) : \theta,\varphi \in [0,\pi]}
\]
is called the $j$th band of $\sigma(H)$.
\end{theorem}

\begin{proof}
This result is standard; for a proof in the discrete setting, see \cite[Theorem~3.3]{KrugPreprint}.
\end{proof}

As a consequence of Proposition~\ref{p:1dlaplacianeigmults}, we can compute the multiplicities of the eigenvalues of $\Delta_{0,0}^\Gamma$ and $\Delta_{\pi,\pi}^\Gamma$ explicitly on square lattices. Here and throughout, we use $\Delta_{\theta,\varphi}^\Gamma$ to denote $H_{\theta,\varphi}^\Gamma$ with $V \equiv 0$.

\begin{lemma}\label{lem:2deigmults}
Let $\Gamma = \Gamma_{(r,r)}$ with $r$ even.  Then
\begin{enumerate}
\item[{\rm(a)}] $\pm 4$ are simple eigenvalues of $\Delta_{0,0}^\Gamma$;
\item[{\rm(b)}] $0$ is an eigenvalue of $\Delta_{0,0}^\Gamma$ with multiplicity congruent to two modulo four;
\item[{\rm(c)}] every other eigenvalue of $\Delta_{0,0}^\Gamma$ has multiplicity divisible by four;
\item[{\rm(d)}] every eigenvalue of $\Delta_{\pi,\pi}^\Gamma$ has multiplicity divisible by four.
\end{enumerate}
\end{lemma}

\begin{proof}
These observations follow immediately from Proposition~\ref{p:1dlaplacianeigmults}.
\end{proof}

We will frequently have recourse to a specific implication of this Lemma:
all eigenvalues of $\Delta_{\pi,\pi}^\Gamma$ have even multiplicity,
while all eigenvalues of $\Delta_{0,0}^\Gamma$ have even multiplicity 
\emph{except the extreme eigenvalues $\pm 4$, which are simple}.

\medskip
One may identify $\Hi_\Gamma$ with $\C^{pq}$ via the ``vectorization'' map $\vect : \Hi_\Gamma \to \C^{pq}$ defined by
\[
\psi \mapsto \vec v = \vect(\psi),
\qquad
\vec v_{kq+\ell+1}
=
\psi_{k,\ell}
\mbox{\ \ for }
(k,\ell) \in \Gamma.
\] 
In view of this identification, $H_{\theta,\varphi}^\Gamma$ enjoys the matrix representation
\[
\vect\circ H_{\theta,\varphi}^\Gamma \circ \vect^{-1}
=
\begin{bmatrix}
H_\varphi^{(0)} & \I &&& e^{-i\theta} \I \\ 
\I & H_\varphi^{(1)} & \I && \\
& \ddots & \ddots & \ddots & \\
&& \I & H_\varphi^{(p-2)} & \I \\
e^{i\theta} \I &&& \I & H_\varphi^{(p-1)}
\end{bmatrix} \in \C^{pq\times pq},
\]
where
\[
H_\varphi^{(j)}
=
\begin{bmatrix}
V_{j,0} & 1 &&& e^{-i\varphi} \\ 
1 & V_{j,1} & 1 && \\
& \ddots & \ddots & \ddots & \\
&& 1 & V_{j,q-2} & 1 \\
e^{i\varphi}  &&& 1 & V_{j,q-1},
\end{bmatrix} \in \C^{q\times q}.
\]

\section{Proof of Theorem} \label{sec:discreteBSC}

\subsection{Proof Strategy}

Given $\p = (p,q)$, let $\Gamma = \Gamma_\p$ and view $\Delta$ as a $\p$-periodic operator. For each $(\theta,\varphi) \in \B \eqdef [0,\pi]^2$, we denote the eigenvalues of $\Delta^\Gamma_{\theta,\varphi}$ by
\[
\lambda_1^\Gamma(\theta,\varphi)
\le 
\cdots
\le 
\lambda_{pq}^\Gamma(\theta,\varphi).
\]
The $j$th \emph{band} of $\sigma(\Delta)$ is then
\[
B_j^\Gamma
=
\set{\lambda_j^\Gamma(\theta,\varphi) : (\theta,\varphi) \in \B}.
\]
Of course, we can eschew the calculation of bands and compute the spectrum of $\Delta$ directly, as it is diagonalized by the Fourier transform on $\Z^2$. One can check that
\[
\sigma(\Delta) = [-4,4].
\]
Then, in view of Theorem~\ref{t:2dperspec}, we deduce that
\begin{equation} \label{eq:laplacebandstouch}
\max B_j^\Gamma
\geq
\min B_{j+1}^\Gamma
\text{ for every }
1 \le j < pq.
\end{equation}

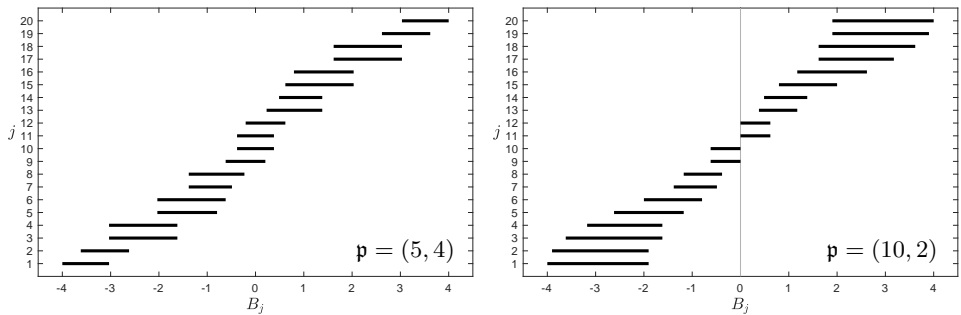
\begin{figure}[b!]
\includegraphics[scale=0.39]{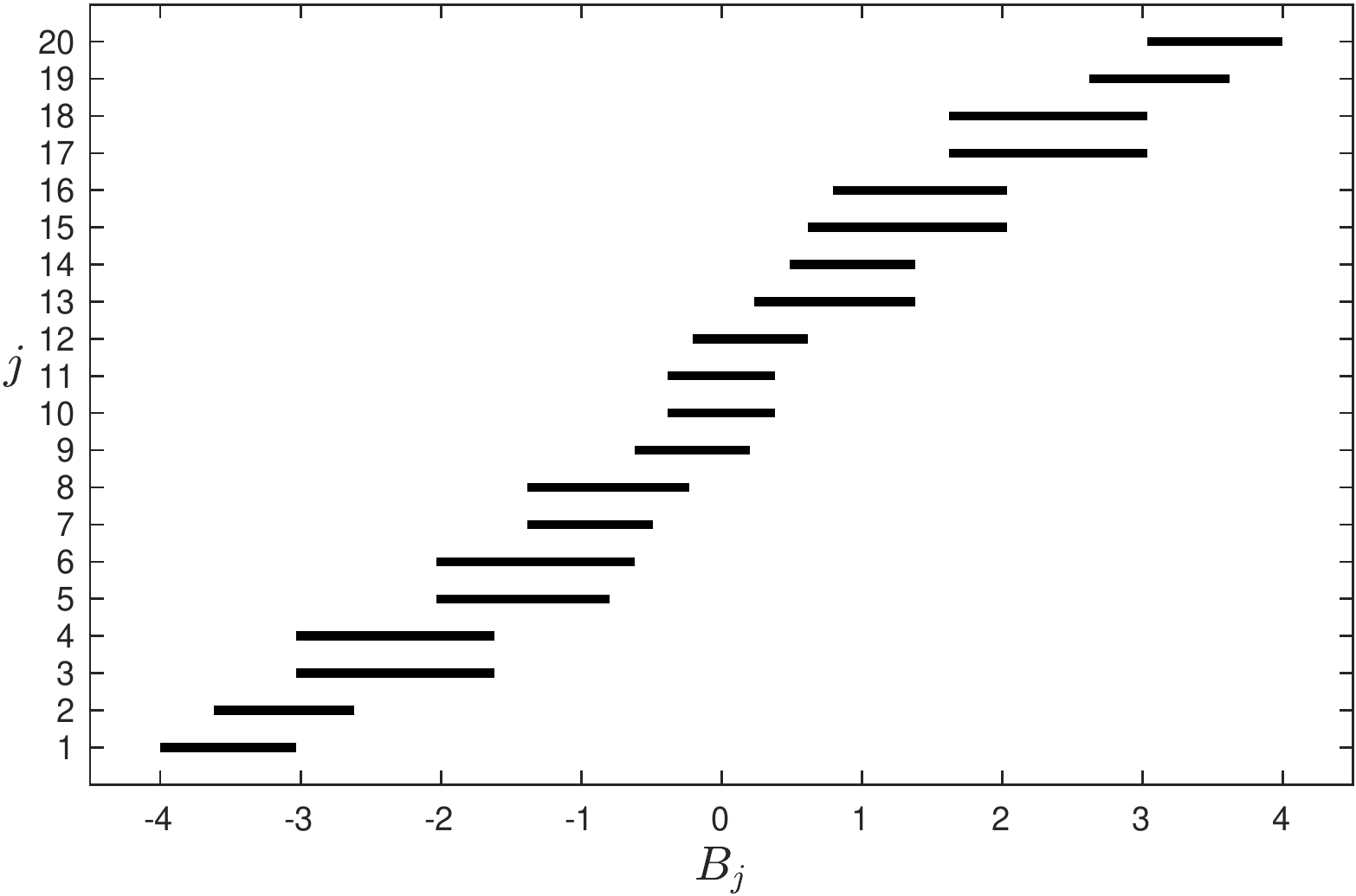}\ \  
\includegraphics[scale=0.39]{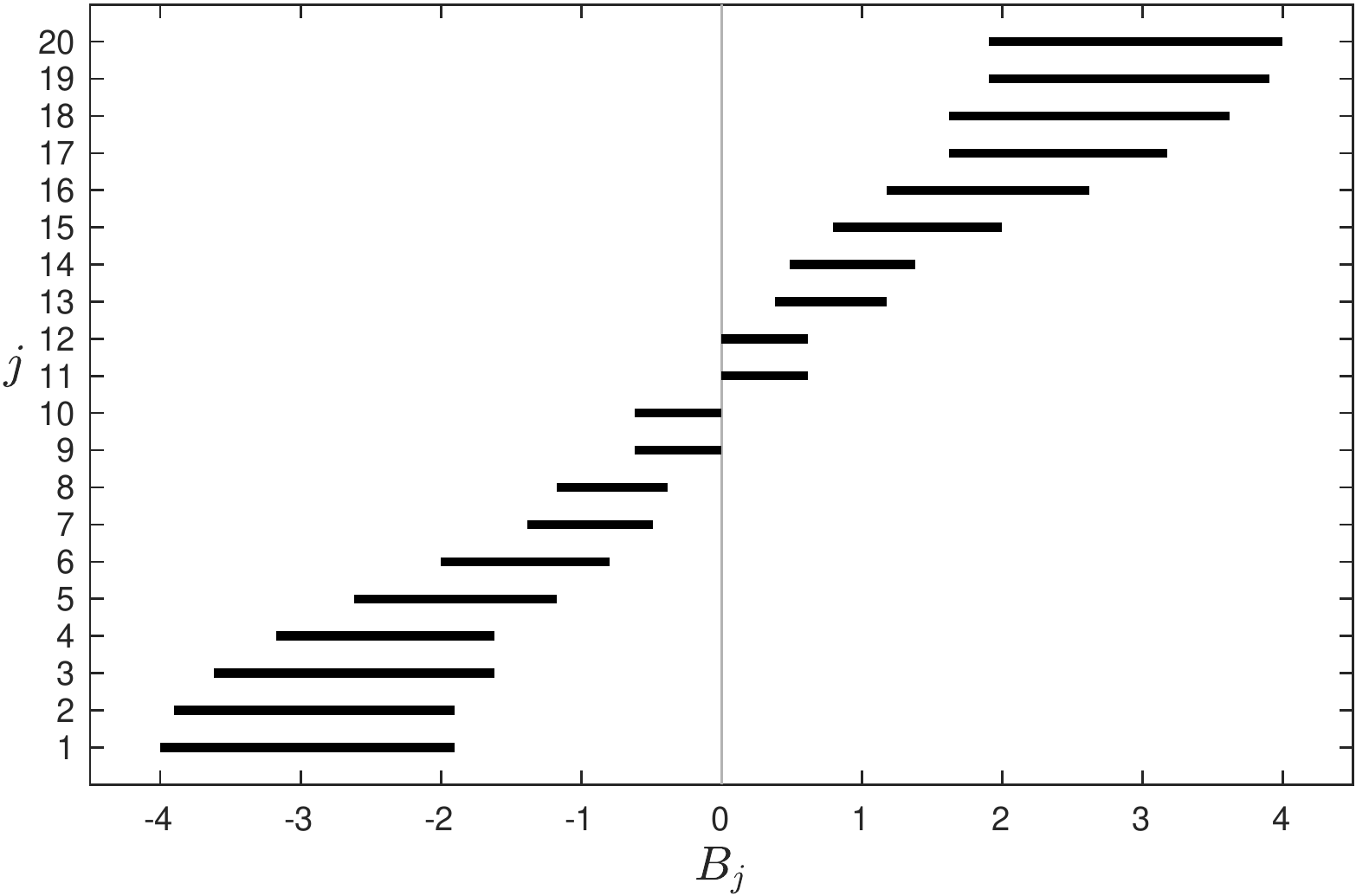}

\begin{picture}(0,0)
\put(-48,35){\footnotesize $\p=(5,4)$}
\put(130,35){\footnotesize $\p=(10,2)$}
\end{picture}

\vspace*{-20pt}
\caption{\label{fig:bands}
The twenty spectral bands $B_j$ whose union gives $\sigma(\Delta)=[-4,4]$
for $\p=(5,4)$ (left) and $\p=(10,2)$ (right).
If $p$ or $q$ is odd, each $E\in(-4,4)$ is in the interior of some $B_j$.
If $p$ and $q$ are even, $E=0$ is only only point in $(-4,4)$ that is not
in the interior of any $B_j$.}
\end{figure}

One key fact that we will use frequently is that $\Delta_{\theta,\varphi}^\Gamma$ has the structure of a \emph{separable} operator. Conceretely, under the natural identification $\Hi_\Gamma \cong \C^p \otimes \C^q$,
\[
\Delta_{\theta,\varphi}^\Gamma
\cong
\Delta_\theta^p \otimes\I_{q\times q} + \I_{p \times p} \otimes \Delta_\varphi^q.
\]
Thus, for every $1 \le j \le pq$, $\lambda_j^\Gamma(\theta,\varphi)$ is of the form $\lambda_{k_1}^p(\theta) + \lambda_{k_2}^q(\varphi)$ for some choice of $1\le k_1 \le p$ and $1 \le k_2 \le q$.

By standard eigenvalue perturbation theory for Hermitian matrices, the edges of the bands are 1-Lipschitz functions of the underlying potential (viewed as an element of $\R^\Gamma$ equipped with the uniform norm topology; compare \cite[Lemma~3.9]{KrugPreprint}). Heuristically, this means that if $\lambda$ is small and $V$ is $\p$-periodic, then $H = \Delta + \lambda V$ can be viewed as an infinitesimal perturbation to the operator $\Delta$ that preserves the number of bands but infinitesimally shifts their locations. The only place such a perturbation can open a gap is at the interface between bands, not in the interior of any band. Thus, to prove Theorem~\ref{t:discreteBSconj}, it suffices to show that every energy in $(-4,4)$ is in the interior of at least one spectral band of $\Delta$ (by compactness); see Figure~\ref{fig:bands}.    

Away from energy zero, it will be most convenient to treat periods $\p$ that are equal and even, so one may introduce $r = \mathrm{lcm}(p,q,2)$, and denote $\fr = (r,r)$ throughout this section. Obviously, any $\p$-periodic potential is also $\fr$-periodic, and hence it suffices to show that any $E \neq 0$ is in the interior of at least one $B_j^{\Gamma_\fr}$. When $\p$ is not even, we must show a similar result for $E=0$. In this case it will be technically convenient for at least one period to be divisible by four, so if $\p$ is not even, introduce
\[
\p'
=
(p',q')
=
\begin{cases}
(p,\mathrm{lcm}(q,4)), & \mbox{if $p$ is odd;}\\
(\mathrm{lcm}(p,4),q), & \mbox{otherwise.}
\end{cases}
\]
Of course, any $\p$-periodic potential is also $\p'$ periodic. Finally, we note that interchanging the roles of $p$ and $q$ leaves the bands invariant, and hence there is no generality lost in assuming that $p$ is odd.
\smallskip

In view of the foregoing discussion, our goal in this section is to prove the following theorem.

\begin{theorem} \label{t:nogaps}
Let $\Gamma = \Gamma_{(r,r)}$ with $r \in \Z_+$ even. For every $E \in (-4,4) \setminus \set{0}$, 
\[
E \in \Int\, B_j^\Gamma \text{ for some } 1 \le j \le r^2.
\]
If $\p = (p,q)$ with $p$ odd and $q$ divisible by four, then
\[
0 \in \Int \, B_j^{\Gamma_\p} \text{ for some } 1 \le j \le pq.
\]

\end{theorem}

We will prove this result in a sequence of steps. Away from a suitable exceptional set, a soft eigenvalue counting argument will show that $E$ must be in the interior of at least one band. The exceptional set corresponds to the energies that occur at the corners of the Brillouin zone, which themselves correspond to sums of eigenvalues of truncations of the corresponding one-dimensional Laplacian.  Let $\Gamma$ and $r$ be as in the statement of the theorem, put
\[
\CC_r
=
\set{2\cos\left( \frac{\pi j}{r} \right) : j \in \Z, \text{ and } 0 \le j \le r} 
= \sigma(\Delta_0^r) \cup \sigma(\Delta_\pi^r),
\]
and define the set of \emph{exceptional energies} to be 
\[
\CC_\Gamma
=
\CC_r + \CC_r
=
\set{a+b : a ,b \in \CC_r}.
\]

\begin{figure}[t!]
\hspace*{-20pt}
\includegraphics[scale=0.38]{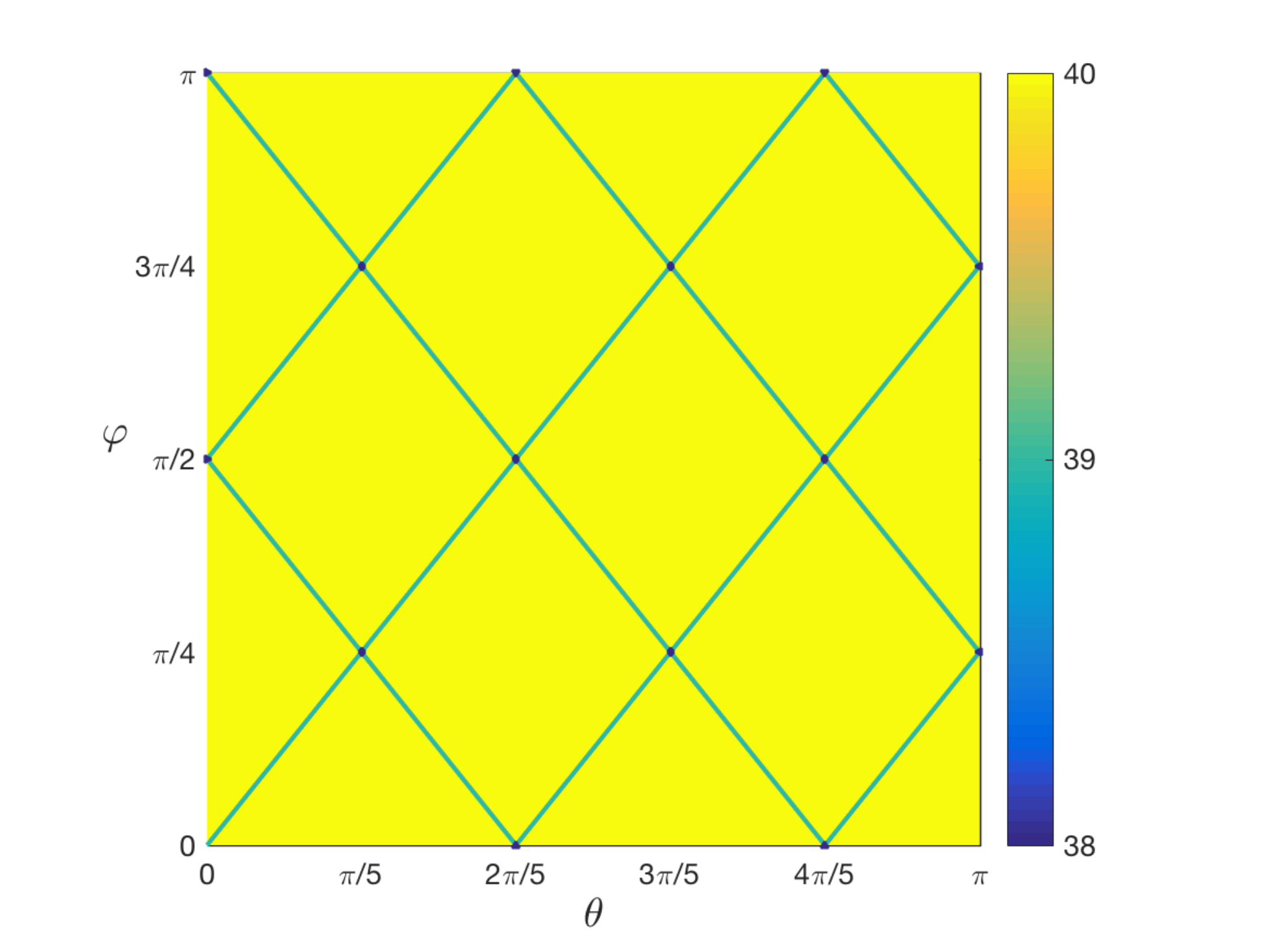}\hspace*{-25pt}
\includegraphics[scale=0.38]{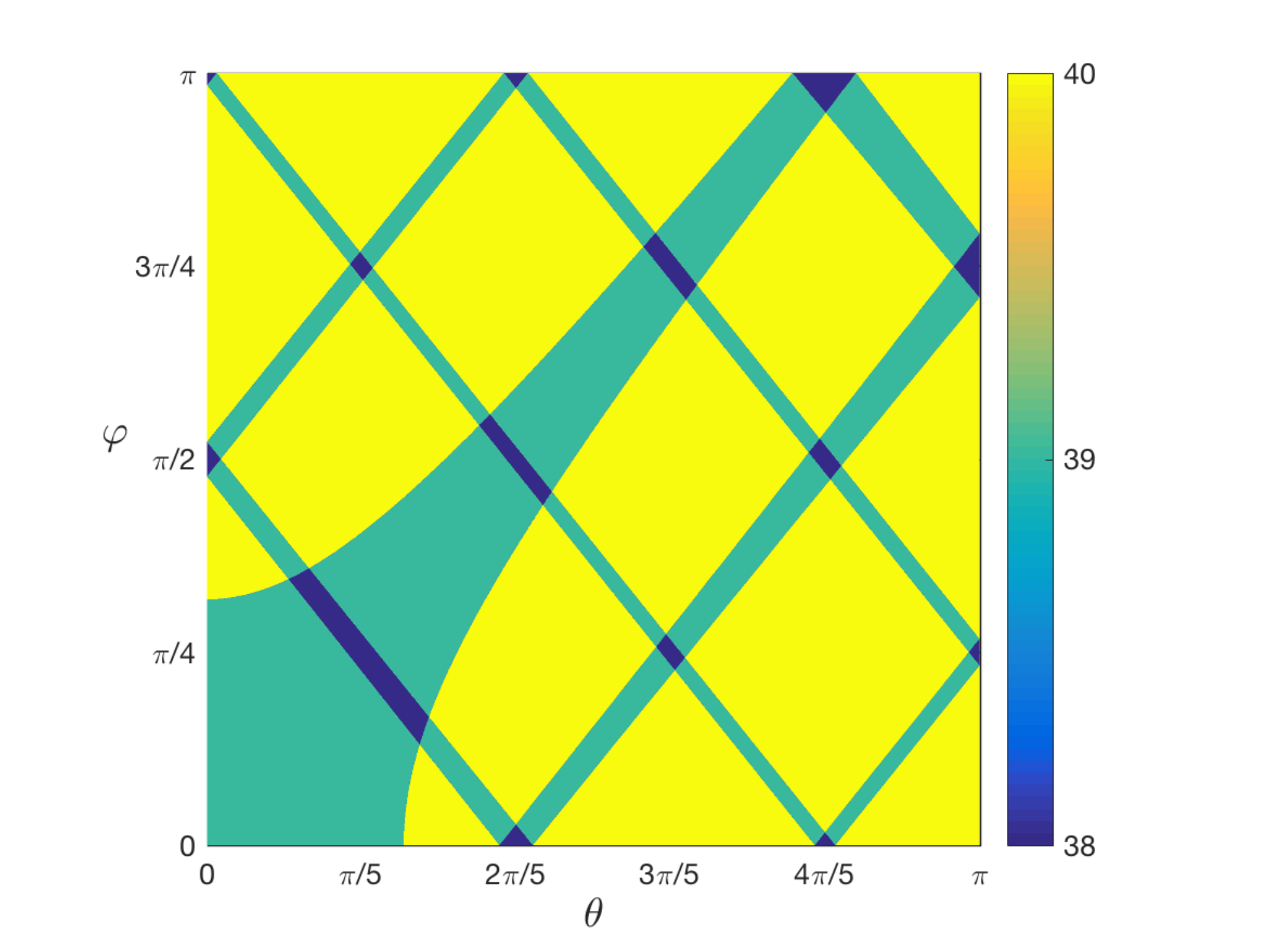}

\vspace*{-10pt}
\caption{\label{fig:quilt_even}
Number of eigenvalues of $\Delta^\Gamma_{\theta,\varphi}$ strictly below 
$E=0 \in \CC_\Gamma$ (left) and $E=-0.01\not\in\CC_\Gamma$ (right) for $\p=(8,10)$.}
\end{figure}

\begin{figure}[t!]
\hspace*{-20pt}
\includegraphics[scale=0.38]{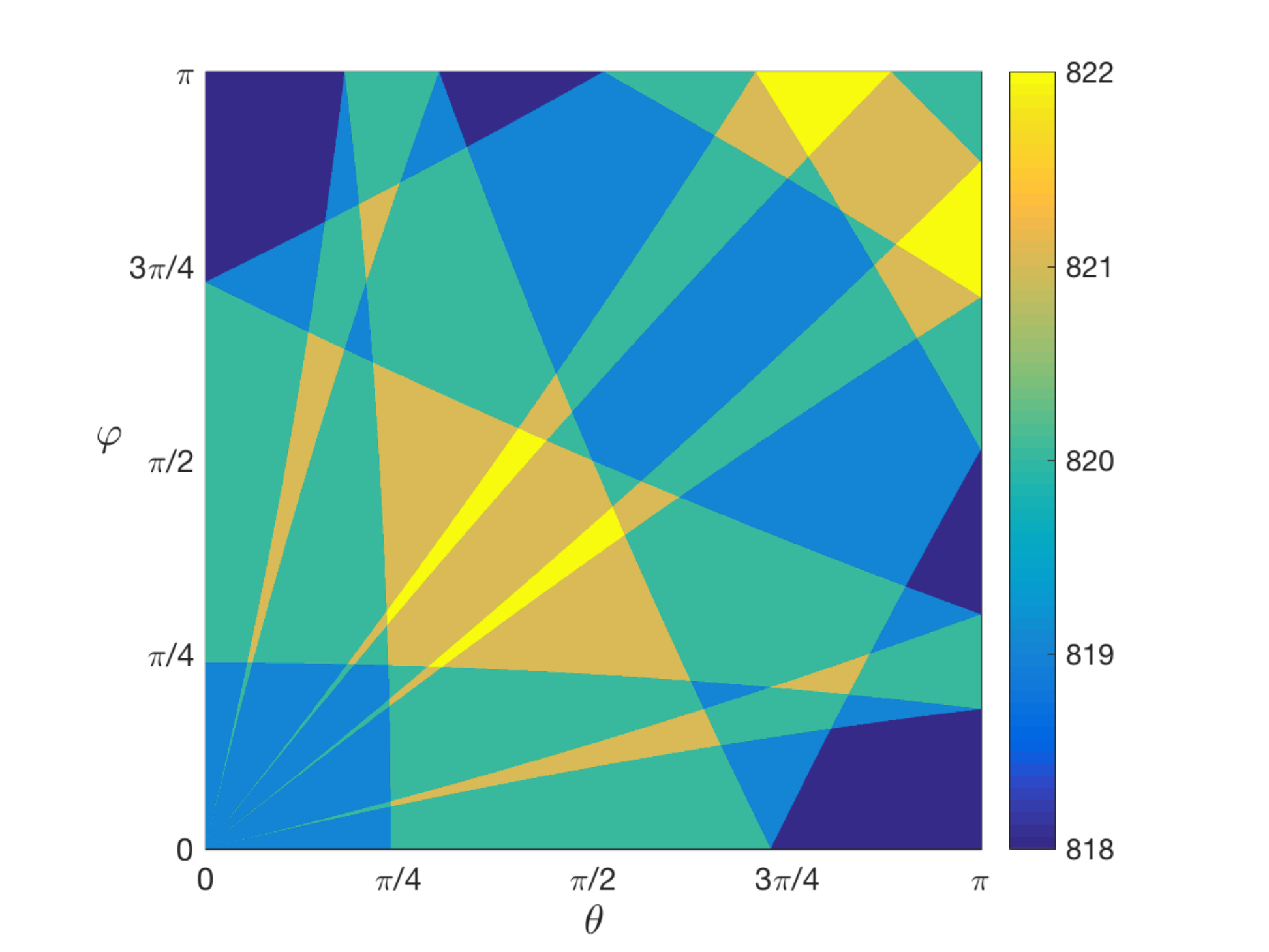}\hspace*{-25pt}
\includegraphics[scale=0.38]{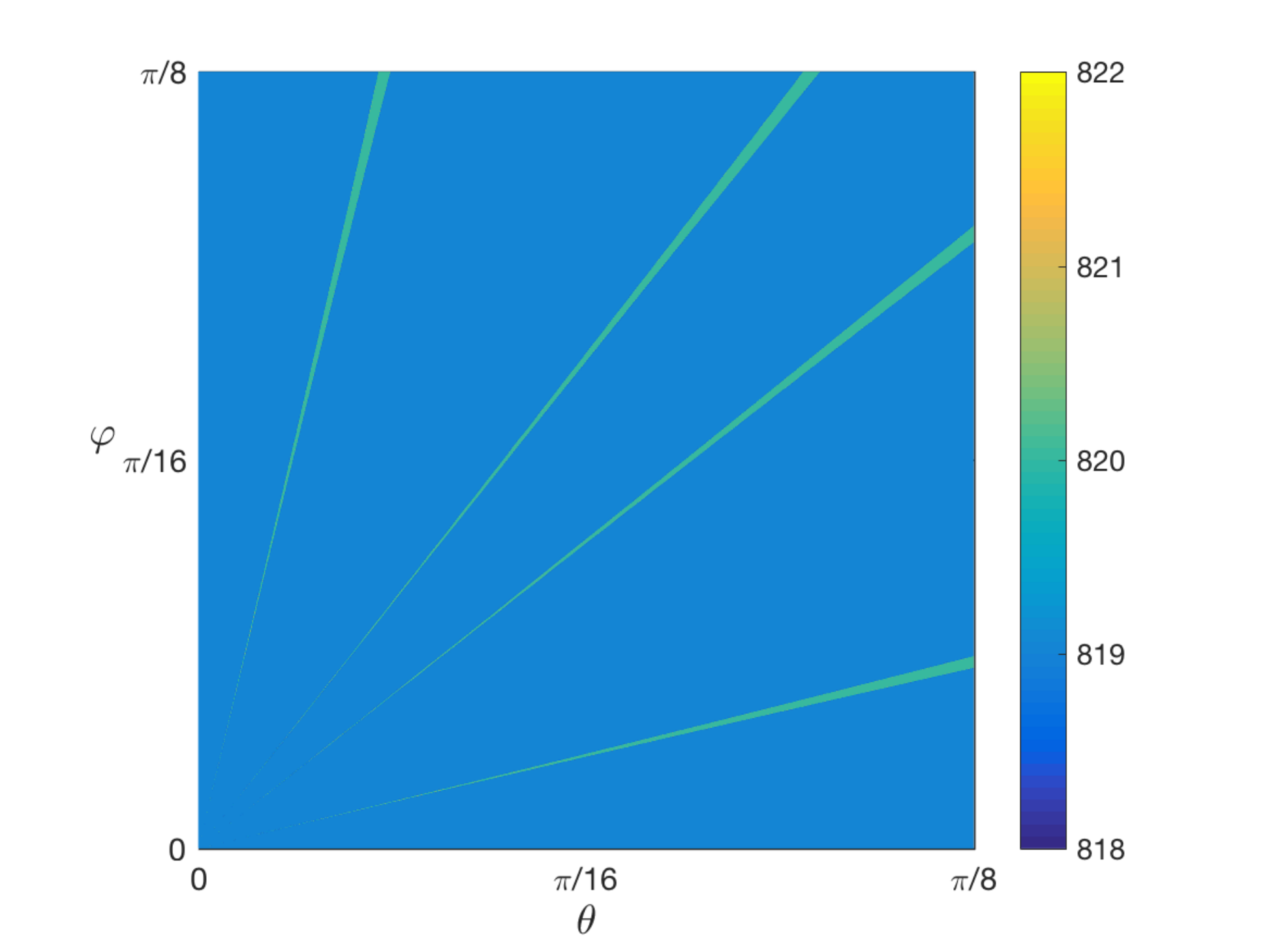}

\vspace*{-10pt}
\caption{\label{fig:quilt_nzE}
Number of eigenvalues of $\Delta^\Gamma_{\theta,\varphi}$ strictly below 
$E=2\cos(\pi/3)+2\cos(\pi/15)\in\CC_\Gamma$ for $\p=(30,30)$ (left plot), 
and a closer look at the lower-left corner (right plot).}
\end{figure}

The key idea is to construct pairs of points in the Brillouin zone with different ``eigenvalue counts.'' More precisely, for each $E$, we want to find $(\theta_1,\varphi_1), (\theta_2,\varphi_2) \in \B$ with $ E \notin \sigma(\Delta^\Gamma_{\theta_1,\varphi_1}) \cup \sigma(\Delta^\Gamma_{\theta_2,\varphi_2})$ and
\begin{equation} \label{eq:diffevcounts}
\#\left( \sigma(\Delta^\Gamma_{\theta_1,\varphi_1}) \cap (-\infty,E) \right)
\ \neq\ 
\# \left( \sigma(\Delta^\Gamma_{\theta_2,\varphi_2}) \cap (-\infty,E) \right).
\end{equation}
Figures~\ref{fig:quilt_even}--\ref{fig:quilt_nzE} suggest the overall strategy. First, the defintion of the exceptional energies is such that $\sigma(\Delta_{0,0}^\Gamma), \sigma(\Delta_{\pi,\pi}^\Gamma) \subseteq \CC_\Gamma$. The right plot in Figure~\ref{fig:quilt_even} illustrates that \eqref{eq:diffevcounts} holds for $(\theta_1,\varphi_1) = (0,0)$ and $(\theta_2,\varphi_2) = (\pi,\pi)$ for $E = -0.01 \in (-4,4)\setminus \CC_\Gamma$. Indeed, it is simple to show that the eigenvalue counts differ for $(0,0)$ and $(\pi,\pi)$ whenever $E\in(-4,4)\setminus \CC_\Gamma$: Lemma~\ref{lem:2deigmults} implies that the left-hand side of \eqref{eq:diffevcounts} is odd while the right-hand side is even!  (See Proposition~\ref{prop:nogap:nonexceptionalpts}.)

For nonzero exceptional energies, $E \in \CC_\Gamma \setminus \{0\}$, the question is more delicate, as $E$ may be an eigenvalue of $\Delta^\Gamma_{0,0}$ or $\Delta_{\pi,\pi}^\Gamma$ having high multiplicity.  Consider $\p=(30,30)$, shown in Figure~\ref{fig:quilt_nzE} for an $E\in \CC_\Gamma$.  While the values 
$(\theta_1,\varphi_1) = (0,0)$ and $(\theta_2,\varphi_2) = (\pi,\pi)$
fail to satisfy~\eqref{eq:diffevcounts}, one can
``push off'' of these corners to nearby $(\theta_j,\varphi_j)$ that satisfy~\eqref{eq:diffevcounts}. The key point is to watch how the eigenvalues of $\Delta_{\theta_j,\varphi_j}^\Gamma$ split, and to perturb $(\theta_j,\varphi_j)$ in a way that avoids the ``tendrils'' visible in the lower-left of the plots in Figure~\ref{fig:quilt_nzE}.

Finally, one has to deal with $E = 0$ when $p$ is odd and $q$ is divisible by four. In this case, we can perturb around the point $(\pi/2,0)$.

The key in making precise the empirical observations from the pictures is furnished by applying Lemma~\ref{lem:cornerderivatives} and a bit of degenerate eigenvalue perturbation theory. In the following subsection, we supply the details.

\subsection{Proof details}

Throughout this section, we fix notation as in the previous subsection: $\Gamma = \Gamma_{(r,r)}$ with $r$ even.

\begin{prop} \label{prop:nogap:nonexceptionalpts}
If $E \in (-4,4) \setminus \CC_\Gamma$, then $E \in \Int \, B_j^\Gamma$ for some $j$.
\end{prop}

\begin{proof}
Fix $E \in (-4,4) \setminus \CC_\Gamma$. By Proposition~\ref{p:1dlaplacianeigmults}, 
\[
\lambda_1^\Gamma(0,0) = -4,
\qquad
\lambda_1^\Gamma(\pi,\pi) = -4\cos\left(\frac{\pi}{r}\right),
\]
so $\Int\,B_1^\Gamma \supset (-4,-4\cos(\pi/r))$. In fact, it is not hard to see that 
\[
B_1^\Gamma 
= 
[-4,-4\cos(\pi/r)].
\]
Consequently, if
\[
E \in 
\left(-4, -4\cos \left( \frac{\pi}{r} \right)\right),
\]
then $E$ is in the interior of the bottom band, so there is nothing to show. Thus (since $E \notin \CC_\Gamma$), assume
\begin{equation} \label{eq:e:apriori:lb}
E
>
-4\cos \left( \frac{\pi}{r} \right).
\end{equation} 
Now, choose $\ell$ and $m$ maximal such that
\[
\lambda_\ell^\Gamma(0,0) < E,
\qquad
\lambda_m^\Gamma(\pi,\pi) < E.
\]
Notice that \eqref{eq:e:apriori:lb} ensures $\ell$ and $m$ exist. Since $E \notin \CC_\Gamma$, Proposition~\ref{p:1dlaplacianeigmults} gives 
\[
E< \lambda_{\ell+1}^\Gamma(0,0),
\qquad
E < \lambda_{m+1}^\Gamma(\pi,\pi).
\]
Furthermore, notice that $\ell \neq m$: Lemma~\ref{lem:2deigmults} implies that $\ell$ must be odd and $m$ must be even, and hence they are distinct.

It suffices to show that $E$ lies in the interior of $B_j^\Gamma$ for some $j \in \set{\ell,\ell+1,m,m+1}$. Suppose on the contrary that $E$ is not in the interior of any of these bands. As a consequence of \eqref{eq:laplacebandstouch}, it follows that
\begin{equation} \label{eq:absurd:bandposition}
E 
= 
\sup B_\ell^\Gamma 
= 
\sup B_m^\Gamma 
= 
\inf B_{\ell+1}^\Gamma 
= 
\inf B_{m+1}^\Gamma.
\end{equation}
 However, this cannot be. For instance, if $\ell < m$, then the positioning indicated by \eqref{eq:absurd:bandposition} implies that $B_m^\Gamma$ lies to the left of $B_{\ell+1}^\Gamma$, even though $\ell+1 \le m$, which is absurd! The case $m < \ell$ is similarly impossible. Since $\ell \neq m$, it follows that $E$ lies in the interior of a band, and we are done.
\end{proof}

\begin{prop} \label{prop:nogap:exceptionalexceptzero}
If $E \in \CC_\Gamma \setminus\{-4,0,4\}$, then $E \in \Int \, B_j^\Gamma$ for some $j$.
\end{prop}

\begin{proof}
Let $E \in \CC_\Gamma \setminus \set{-4,0,4}$ be given. By symmetry, it suffices to work with $E < 0$. As in the proof of Proposition~\ref{prop:nogap:nonexceptionalpts}, we choose $\ell,m \in \Z_+$ maximal with the property that
\[
\lambda_\ell^\Gamma(0,0) < E,
\qquad
 \lambda_m^\Gamma(\pi,\pi)
<
E.
\]
We adopt the convention that $\lambda_0^\Gamma(\pi,\pi) = -\infty$ to deal with the point
\[
E = -4\cos(\pi/r)
=
\lambda_1^\Gamma(\pi,\pi),
\]
in which case we may take $m = 0$. Arguing as before, we know that $\ell$ is odd and $m$ is even. The challenge is that these exceptional energies correspond to highly degenerate eigenvalues of $\Delta_{0,0}^\Gamma$ and $\Delta_{\pi,\pi}^\Gamma$.  
Our goal is to nudge into an eigenvalue count discrepancy as seen in the proof of Proposition~\ref{prop:nogap:nonexceptionalpts} 
by designing suitable perturbations of $(0,0)$ and $(\pi,\pi)$ that move half of the degeneracies to the left and half to the right.  More precisely, we shall prove the following.
\medskip

\noindent \textbf{Claim.} There exist $(\theta_1,\varphi_1), (\theta_2,\varphi_2) \in \B$ and integers $n_1,n_2 \geq 0$ such that
\[
\lambda_{\ell+2n_1}^\Gamma(\theta_1,\varphi_1) 
< 
E  
< 
\lambda_{\ell+2n_1+1}^\Gamma(\theta_1,\varphi_1),\]
and
\[
\lambda_{m+2n_2}^\Gamma(\theta_2,\varphi_2) 
< 
E 
< 
\lambda_{m+2n_2+1}^\Gamma(\theta_2,\varphi_2).
\]

Before proving the claim, we note that it implies the conclusion of Proposition~\ref{prop:nogap:exceptionalexceptzero} by the same argument that proved Proposition~\ref{prop:nogap:nonexceptionalpts}. Namely, since $\ell$ is odd and $m$ is even, it follows that $\ell+2n_1 \neq m + 2n_2$. Then, if $E$ fails to be in the interior of one of $B_{\ell+2n_1}^\Gamma$, $B_{\ell+2n_1+1}^\Gamma$, $B_{m+2n_2}^\Gamma$, or $B_{m + 2n_2+1}^\Gamma$, we get self-contradictory band locations, just as in the proof of Proposition~\ref{prop:nogap:nonexceptionalpts}. 

\begin{proof}[Proof of Claim]

The idea is to produce $(\theta_1,\varphi_1)$ and $(\theta_2,\varphi_2)$ as perturbations of $(0,0)$ and $(\pi,\pi)$. We obtain $n_1$ and $n_2$ by keeping track of multiplicities carefully. Start with $(\pi,\pi)$, which is simpler. If $E \notin \sigma(\Delta_{\pi,\pi}^\Gamma)$, there is nothing to do: take $n_2 = 0$ and $\theta_2 = \varphi_2 = \pi$. Now, suppose $E \in \sigma(\Delta_{\pi,\pi}^\Gamma)$. As mentioned before, all eigenvalues of $\Delta_{\pi,\pi}^\Gamma$ have multiplicity four by Lemma~\ref{lem:2deigmults}. Thus,
\[
\lambda_m^\Gamma(\pi,\pi) 
<
E
= 
\lambda_{m+1}^\Gamma(\pi,\pi)
=
\cdots
=
\lambda_{m+4s}^\Gamma(\pi,\pi)
<
\lambda_{m+4s+1}^\Gamma(\pi,\pi)
\]
for some integer $s \geq 1$. In fact, we can be more explicit: each quartet of $E$'s arises from a pair of doubly degenerate eigenvalues of $\Delta_\pi^r$, say $\lambda_{2a-1}^r(\pi) = \lambda_{2a}^r(\pi)$ and $\lambda^r_{2b-1}(\pi) = \lambda_{2b}^r(\pi)$ for some integers $a,b$ with
\[
E
=
\lambda_{2a-1}^r(\pi) + \lambda_{2b-1}^r(\pi).
\]
Then, since $r$ is even, Lemma~\ref{lem:cornerderivatives}(b) implies that $\lambda_{2a-1}$ is strictly increasing and $\lambda_{2a}$ is strictly decreasing on $(0,\pi)$. Consequently, we get
\begin{align*}
\lambda_{2a-1}^r(\pi-\e) + \lambda_{2b-1}^r(\pi)
& < E \\
\lambda_{2a-1}^r(\pi-\e) + \lambda_{2b}^r(\pi)
& < E \\
\lambda_{2a}^r(\pi-\e) + \lambda_{2b-1}^r(\pi)
& > E \\
\lambda_{2a}^r(\pi-\e) + \lambda_{2b}^r(\pi)
& > E
\end{align*}
for all $\e > 0$.
Applying this analysis to each pair of degenerate eigenvalues of $\Delta_\pi^r$ that sum to $E$, we see that we may take $n_2 = s$ and $(\theta_2,\varphi_2) = (\pi-\e,\pi)$ for some small $\e > 0$.
\medskip

We now turn to constructing $(\theta_1,\varphi_1)$ and $n_1$. As before, if $E \notin \sigma(\Delta_{0,0}^\Gamma)$, there is nothing to do, so assume $E \in \sigma(\Delta_{0,0}^\Gamma)$.  Moreover, if $E \notin -2 + \sigma(\Delta_0^r)$, then the argument from the $(\pi,\pi)$ case just studied carries over verbatim, so we may take $(\theta_1,\varphi_1) = (\e,0)$ for small $\e > 0$. 
(Figure~\ref{fig:quilt_nzE} shows an example.)
It remains to deal with the case when $E \in -2+\sigma(\Delta_0^r)$, so let such an $E$ with $E\in(-4,0)$ be given. Again, the multiplicity of $E$ as an eigenvalue of $\Delta_{0,0}^\Gamma$ is $4t$ for some $t \in \Z_+$ by Lemma~\ref{lem:2deigmults}.

This case is the most delicate because one has to find a perturbative argument that simultaneously works for three types of pairs of eigenvalues of $\Delta_{0}^r$ that sum to $E$. It turns out that we can take $(\theta_1,\varphi_1) = ((1+\delta)\e,\e)$ for $\delta,\e > 0$ sufficiently small. Let us describe how this comes about.

First, notice that $E \in -2 + \sigma(\Delta_0^r)$ implies that there exists an integer $c$ with $E = \lambda_1(0) + \lambda_{2c}(0)$. Fix $\delta > 0$ small; we will describe how small $\delta$ must be presently. In view of Proposition~\ref{p:1dlaplacianeigmults} and Lemma~\ref{lem:cornerderivatives}, we deduce
\begin{align}
\label{eq:eigsplit01}
\lambda_1(\theta) + \lambda_{2c}(\varphi) & < E \\
\label{eq:eigsplit02}
\lambda_1(\theta) + \lambda_{2c+1}(\varphi) & > E \\
\label{eq:eigsplit03}
\lambda_{2c}(\theta) + \lambda_{1}(\varphi) & < E \\
\label{eq:eigsplit04}
\lambda_{2c+1}(\theta) + \lambda_{1}(\varphi) & > E 
\end{align}
if $(\theta,\varphi) = ((1+\delta)\e,\e)$ for small $\e > 0$. 
We must address the possibility that this $E \in -2 + \sigma(\Delta_0^r)$ can also be expressed as the sum of pairs of doubly degenerate eigenvalues, i.e., that there exist positive integers $a$ and $b$ for which
\[
\lambda_{2a}(0) + \lambda_{2b}(0) = E.
\]
If $a \neq b$, then, since $E \neq 0$, we may assume without loss that $|\lambda_{2a}(0)| < |\lambda_{2b}(0)|$.
Lemma~\ref{lem:cornerderivatives}(c) then implies that
\[ |\lambda_{2b}'(0)| = |\lambda_{2b+1}'(0)| < |\lambda_{2a}'(0)| = |\lambda_{2a+1}'(0)|.\] 
Now take some $\delta$ such that
\[ 0 < \delta <  \frac{|\lambda_{2a}'(0)|}{|\lambda_{2b}'(0)|} - 1.\]
Again take $(\theta,\varphi) = ((1+\delta)\e,\e)$ with $\e > 0$ sufficiently small.
Then Lemma~\ref{lem:cornerderivatives}(b) gives
\begin{align}
\label{eq:eigsplit05}
\lambda_{2a}(\theta) + \lambda_{2b}(\varphi)
& < 
E \\
\label{eq:eigsplit06}
\lambda_{2a+1}(\theta) + \lambda_{2b}(\varphi)
& >
E \\
\label{eq:eigsplit07}
\lambda_{2a}(\theta) + \lambda_{2b+1}(\varphi)
& < 
E \\
\label{eq:eigsplit08}
\lambda_{2a+1}(\theta) + \lambda_{2b+1}(\varphi)
& > 
E \\
\label{eq:eigsplit09}
\lambda_{2b}(\theta) + \lambda_{2a}(\varphi)
& < 
E \\
\label{eq:eigsplit10}
\lambda_{2b+1}(\theta) + \lambda_{2a}(\varphi)
& <
E \\
\label{eq:eigsplit11}
\lambda_{2b}(\theta) + \lambda_{2a+1}(\varphi)
& > 
E \\
\label{eq:eigsplit12}
\lambda_{2b+1}(\theta) + \lambda_{2a+1}(\varphi)
& > 
E.
\end{align}
Notice that this particular choice of $\delta$ is required to ensure the inequalities \eqref{eq:eigsplit10} and \eqref{eq:eigsplit11} hold.
Thus, we ultimately take $\delta > 0$ with
\[
\delta
<
\min \left( \frac{|\lambda_{2a}'(0)|}{|\lambda_{2b}'(0)|} - 1 \right),
\]
where the minimum ranges over all pairs with $\lambda_{2a}(0) + \lambda_{2b}(0) = E$ and $|\lambda_{2a}(0)| < |\lambda_{2b}(0)|$.  On the other hand, when $a=b$, then \eqref{eq:eigsplit05}--\eqref{eq:eigsplit08} still hold for $(\theta,\varphi) = ((1+\delta)\e,\e)$, as long as $\e$ is sufficiently small. (Notice that one only considers four inequalities, since ``interchanging the roles of $a$ and $b$'' would be redundant in this scenario.)

Thus, we take $(\theta_1,\varphi_1) = ((1+\delta)\e,\e)$ for $\e>0$ sufficiently small and $n_1 = t$.  We have established the claim.
\end{proof}
With the claim proved, the proof of Proposition~\ref{prop:nogap:exceptionalexceptzero} is complete.
\end{proof}

\begin{figure}[t!]
\hspace*{-20pt}
\includegraphics[scale=0.38]{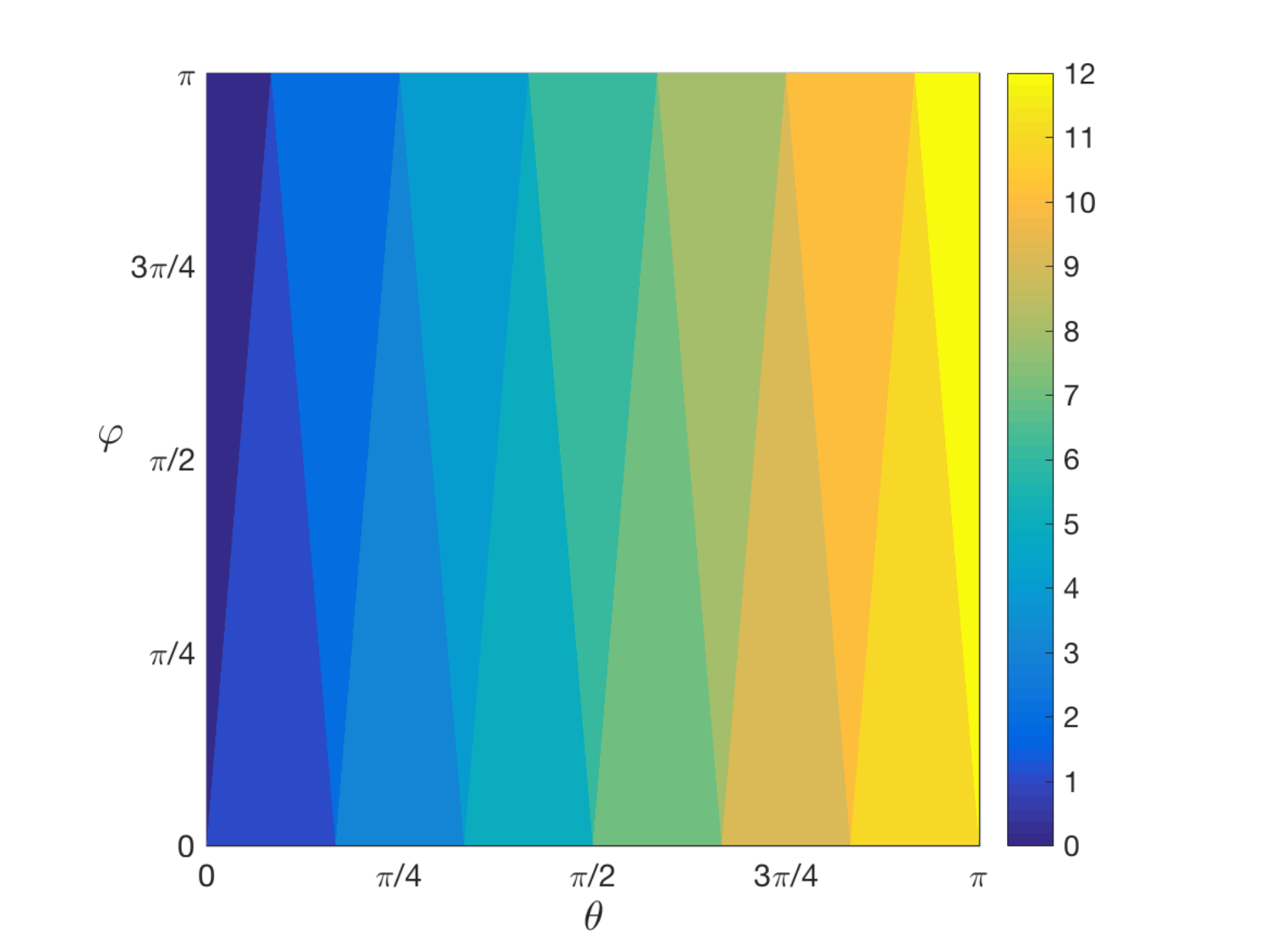}\hspace*{-25pt}
\includegraphics[scale=0.38]{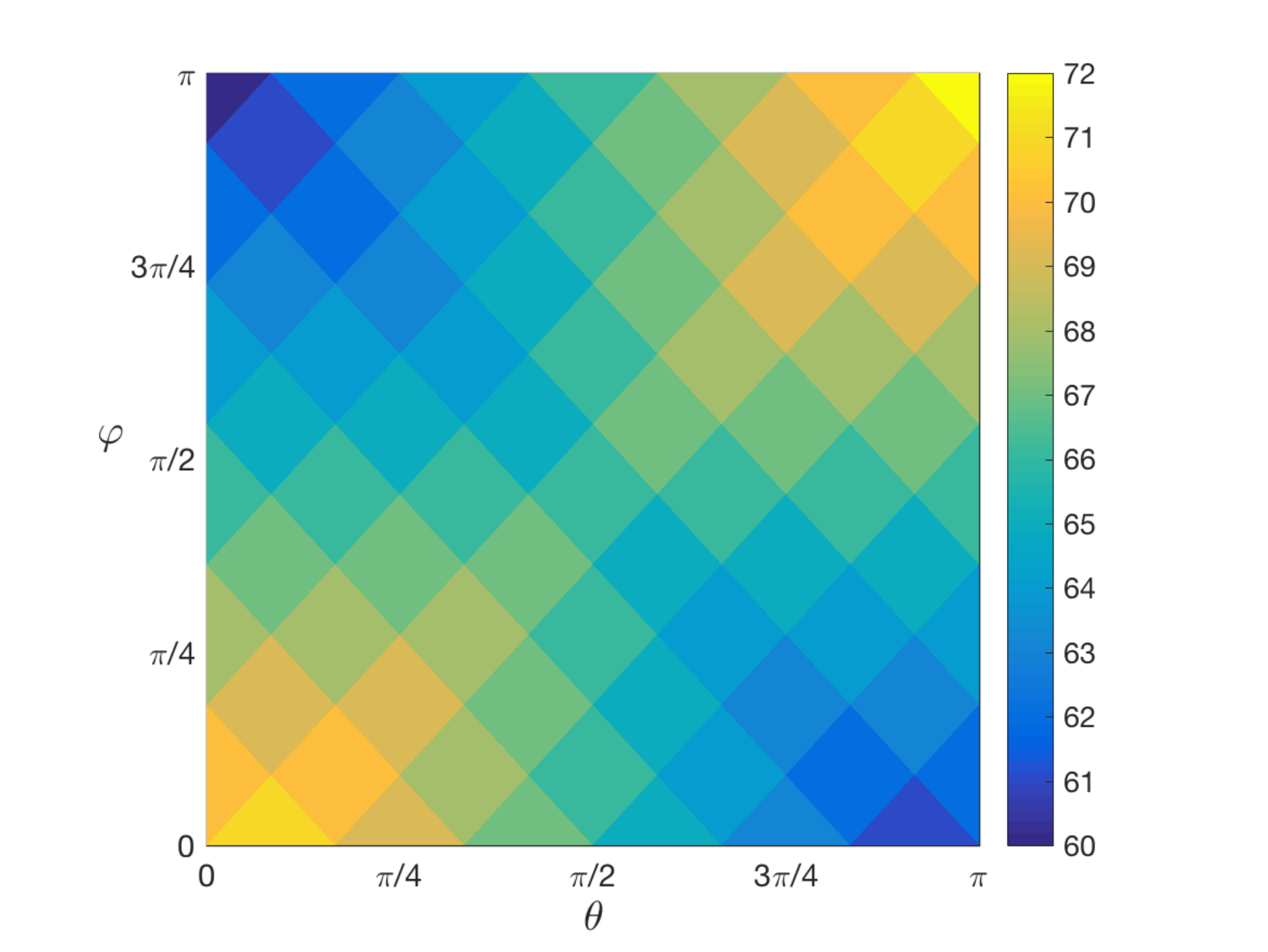}

\vspace*{-10pt}
\caption{\label{fig:quilt_odd0}
Number of eigenvalues of $\Delta^\Gamma_{\theta,\varphi}$ strictly below $E=0$ 
for $\p=(1,12)$ (left plot) and $\p=(11,12)$ (right plot).}
\end{figure}

Finally, we show that $E=0$ is in the interior of a band, provided some period is odd.
The examples in Figure~\ref{fig:quilt_odd0} illustrate this scenario.

\begin{prop} \label{prop:nogap:zero}
If $p$ is odd and $q$ is divisible by four, then $0 \in \Int\,B_j^{\Gamma}$ for some $j$, where $\Gamma = \Gamma_\p$.
\end{prop}

\begin{proof}
Consider the point $(\theta,\varphi) = (\pi/2,0) \in \B$, and observe that
\begin{equation} \label{eq:sum0}
\lambda_{\frac{p+1}{2}}^p(\pi/2) 
= 
\lambda_{\frac{q}{2}}^q(0)
=
\lambda_{\frac{q}{2}+1}^q(0) = 0,
\end{equation}
by Proposition~\ref{p:1dlaplacianeigmults}. In particular, \eqref{eq:sum0} implies that  $0 \in \sigma(\Delta_{\pi/2,0}^{\Gamma})$. 
Proposition~\ref{p:1dlaplacianeigmults} also shows that all the eigenvalues of $\Delta_{\pi/2}^p$ are simple and the spectrum is symmetric about zero:
\[
\lambda_j^p(\pi/2)
=
\lambda_{p+1-j}^p(\pi/2)
\text{ for every } 1 \le j \le p.
\]

To show that zero is in the interior of a band, we will compute the multiplicity of $0$ as an eigenvalue of $\Delta_{\pi/2,0}^{\Gamma}$, and then use Lemma~\ref{lem:cornerderivatives} to understand how this eigenvalue splits as we perturb around $\theta = \pi/2$ (leaving $\varphi = 0$ fixed).
\medskip

Now, let us make some observations. First, since the spectrum of $\Delta_0^q$ is also symmetric about zero (see Proposition~\ref{p:1dlaplacianeigmults}), if $\lambda_j^p(\pi/2) \notin\sigma(\Delta_0^q)$, then there is no $\mu \in \sigma(\Delta_0^q)$ with
\[
\lambda_j^p(\pi/2) + \mu = 0.
\]
On the other hand, for any $j$ such that $\lambda_j^p(\pi/2) \in \sigma(\Delta_0^q)$, there exists $k$ such that
\begin{equation} \label{eq:nogapzero1}
\lambda_j^p(\pi/2) + \lambda_{2k}^q(0)
=
\lambda_j^p(\pi/2) + \lambda_{2k+1}^q(0)
=
0.
\end{equation}
By symmetry,
\begin{equation} \label{eq:nogapzero2}
\lambda_{p+1-j}^p(\pi/2) + \lambda_{q-2k}^q(0)
=
\lambda_{p+1-j}^p(\pi/2) + \lambda_{q+1-2k}^q(0)
=
0.
\end{equation}
Of course, the pairs in \eqref{eq:nogapzero1} and \eqref{eq:nogapzero2} are distinct whenever $j \neq \frac{p+1}{2}$. Finally, we have from~(\ref{eq:sum0}) that $E=0$ can
emerge as an eigenvalue in one additional way:
\[
\lambda_{\frac{p+1}{2}}^p(\pi/2) + \lambda_{\frac{q}{2}}^q(0) =
\lambda_{\frac{p+1}{2}}^p(\pi/2) + \lambda_{\frac{q}{2}+1}^q(0)
=
0.
\]
Taken together, these observations imply that the multiplicity of zero as an eigenvalue of $\Delta_{\pi/2,0}^{\Gamma}$ is congruent to two modulo four, say
\[
\lambda_s^{\Gamma}(\pi/2,0)
<
0
=
\lambda_{s+1}^{\Gamma}(\pi/2,0)
=
\cdots
=
\lambda_{s+4t+2}^{\Gamma}(\pi/2,0)
<
\lambda_{s+4t+3}^{\Gamma}(\pi/2,0)
\]
for integers $s,t$. Since $p$ is odd, it is congruent to one of $\pm 1$ modulo four. The two cases are very similar, so let us assume $p \equiv 1\mod 4$. Then, by Lemma~\ref{lem:cornerderivatives}, $\lambda_{(p+1)/2}'(\pi/2)<0$ and  so
\begin{equation}\label{eq:nogapzero3}
\lambda_{\frac{p+1}{2}}^p\left(\frac{\pi}{2}-\e\right)
>
0
>
\lambda_{\frac{p+1}{2}}^p\left(\frac{\pi}{2}+\e\right)
\end{equation}
for $\e>0$ small. Additionally, since $p$ is odd, we notice that $p-j$ and $p-(p+1-j)$ have the same parity for all $j$. Consequently, Lemma~\ref{lem:cornerderivatives} implies that the derivatives of $\lambda_j^p$ and $\lambda_{p+1-j}^p$ have the same sign (and both derivatives are nonzero when evaluated at $\theta = \pi/2$).  Thus, for each pair $j,k$ for which $j \neq (p+1)/2$ and \eqref{eq:nogapzero1} and \eqref{eq:nogapzero2} hold, we get a quartet of eigenvalues of $\Delta_{\pi/2,0}^{\Gamma}$ that all cross zero in the same direction as $\theta$ crosses $\pi/2$. Combining this observation with \eqref{eq:nogapzero3}, we get integers $\ell_+ \neq \ell_-$ so that
\[
\lambda_{s + \ell_\pm}^{\Gamma} \left(\frac{\pi}{2}\pm\e,0\right)
<
0
<
\lambda_{s + \ell_\pm + 1}^{\Gamma} \left(\frac{\pi}{2}\pm\e,0\right)
\]
for small $\e>0$. In particular, the arguments above imply that $\ell_-$ is divisible by four and $\ell_+$ is congruent to two modulo four. Hence, with $n_\pm = s + \ell_\pm$, we may argue as before to see that $0$ is in the interior of at least one of $B_{n_+}^{\Gamma}$, $B_{n_-}^{\Gamma}$, $B_{n_+ + 1}^{\Gamma}$, or $B_{n_- + 1}^{\Gamma}$.

\end{proof}

\begin{remark}
Figure~\ref{fig:quilt_odd0} gives two specific examples of the scenario treated in Proposition~\ref{prop:nogap:zero}.
For $\p = (1,12)$, we have the $p \equiv 1\mod 4$ case detailed in the proof, with $s=5$, $t=0$, $\ell_- = 0$, and $\ell_+ = 2$.
For $\p = (11,12)$, we have $p \equiv -1 \mod 4$ and $s=65$, $t=0$, $\ell_- = 2$, and $\ell_+ = 0$.
\end{remark}

\begin{remark}
Passing from $q$ to $\mathrm{lcm}(q,4)$ for Proposition~\ref{prop:nogap:zero} is not necessary, but it is convenient, as it saves us an argument by cases. Given $\p = (p,q)$ with $p$ odd, one can run the argument above by perturbing around $(\pi/2,\pi)$ when $q \equiv 2\mod 4$ and perturbing around $(\pi/2,\pi/2)$ when $q$ is odd, which obviates the need to alter the vertical period.
\end{remark}

Taken together, these propositions prove the desired result.

\begin{proof}[Proof of Theorem~\ref{t:discreteBSconj}]
This follows immediately from Theorem~\ref{t:nogaps}, which itself follows from Propositions~\ref{prop:nogap:nonexceptionalpts}, \ref{prop:nogap:exceptionalexceptzero}, and \ref{prop:nogap:zero}.
\end{proof}

\end{document}